\documentclass{elsarticle}

\usepackage{lineno,hyperref}
\usepackage{graphicx}
\usepackage{multirow}
\usepackage{textcomp}
\usepackage{amssymb}
\usepackage{epstopdf}
\usepackage{amsmath,amsthm,amscd,amssymb}
\usepackage{latexsym}
\usepackage{bm}
\usepackage{geometry}                % See geometry.pdf to learn the layout options. There are lots.
\usepackage{mathtools}
\usepackage{centernot}
\usepackage{bbm}
\usepackage{enumitem}
\usepackage{subcaption}
\numberwithin{equation}{section}

\theoremstyle{plain}
\newtheorem{theorem}{Theorem}[section]
\newtheorem{lemma}[theorem]{Lemma}

\newtheorem{proposition}[theorem]{Proposition}

\theoremstyle{definition}
\newtheorem{definition}[theorem]{Definition}

\theoremstyle{remark}
\newtheorem{remark}[theorem]{Remark}

\numberwithin{equation}{section}

\newlist{Properties}{enumerate}{2}
\setlist[Properties]{label=\textbf{R.\arabic*},itemindent=*}

\def\bb #1{ {\mathbb #1} }
\newcommand{\1}{\mathbbm{1}}
\def\ds{\displaystyle}
\DeclareMathOperator{\rank}{rank}
\DeclareMathOperator{\SPAN}{span}

\begin{document}

%\author{M. Bennett, D. Hart, A. Iosevich, J. Pakianathan and M. Rudnev}

\date{today}

%\thanks{I would like to thank the anonymous review}

\begin{frontmatter}

\title{Bounds on sizes of generalized caps in $AG(n,q)$ via the \\Croot-Lev-Pach polynomial method}

\author{Michael Bennett}

\address{111 Harold Ave., Cornwall, NY 12518 \\ mike.b.bennett87@gmail.com}

\begin{abstract} 

In 2016, Ellenberg and Gijswijt employed a method of Croot, Lev, and Pach to show that a maximal cap in $AG(n,q)$ has size $O(q^{cn})$ for some $c < 1$. In this paper, we show more generally that if $S$ is a subset of $AG(n,q)$ containing no $m$ points on any $(m-2)$-flat, then $|S| < q^{c_mn}$ for some $c_m < 1$, as long as $q$ is odd or $m$ is even.

\end{abstract}

\begin{keyword}
polynomial method \sep affine caps \sep function rank \sep finite geometry

\end{keyword}

\end{frontmatter}

%\vspace{.15 in}

\section{Introduction and main theorem}

Let $q$ be a power of a prime. A cap is a set of points in the projective geometry $PG(n,q)$, no three of which lie on a common line. A cap $A$ is maximal if for any other cap $B$, $|B| \le |A|$, and we denote the size of a maximal cap in $PG(n,q)$ by $m_2(n,q)$. Caps may be similarly defined in the affine space $AG(n,q)$. The problem of finding maximal caps has been studied extensively in both types of spaces (see for instance, \cite{Pot} or \cite{Thas}). One of the primary motivations behind the study of caps is their application to coding theory. See, for instance, section 17.2 of \cite{bierbook} for a detailed explanation of the connection between caps and linear codes.

One question that arises in the investigation of maximal caps is how they grow with $n$. In particular, we would like to estimate
$$
\mu(q) = \limsup_{n \to \infty} \frac{\log_q\left(m_2(n,q)\right)}{n}.
$$
While we will be working exclusively in affine space in this paper, note that if $A$ is a maximal cap in $AG(n,q)$, then a maximal cap in $PG(n,q)$ has at most $|A|(1 + o(1))$ points. Therefore, any bounds on $\mu(q)$ apply to both affine and projective space. Trivially, we have $\mu(q) \le 1$, and a lower bound of $2/3$ can be achieved quite easily: it is well know that a maximal cap $P \subset AG(3,q)$ has $q^2$ points (see, for instance, \cite{EFLS}). Then $P^k \subset AG(3k,q)$ is a cap of $q^{2k}$ points.

Recently, the problem of finding better estimates for $\mu(3)$ has been of great interest. It was long suspected that $\mu(3) < 1$, but it took some time to find an appropriate method of attack. In 1985, Meshulam (\cite{Mesh}) proved that the maximum size of a cap in $AG(n,3)$ (sometimes called a ``cap set") is $\frac{2}{n} \cdot 3^n$ using Fourier techniques. In 2011, Bateman and Katz  (\cite{BK}) combined these Fourier techniques with spectral methods to improve this bound to $\ds O \left( 3^n/n^{1 + \epsilon}\right)$, where $\epsilon$ is independent of $n$. It was not until 2016 that Ellenberg and Gijswijt  (\cite{ElGi}) used a polynomial method developed by Croot, Lev, and Pach (\cite{CLP}) to get $m_2(n,3) = O(2.756^{n})= O(3^{0.923n})$, and hence $\mu(3) <1$. In fact, they show that $\mu(q) <1$ for all $q$. This rendered the problem essentially solved; however, the bounds they achieve are not known to be sharp. The best known lower bound in the $q=3$ case is $\mu(3) > 0.724$, due to Edel (\cite{Edel}). 

In 2001, Hirschfeld and Storme collected many of the best known bounds on maximal caps in $PG(n,q)$ at the time. While they are nontrivial, one can see in \cite{Hirsch} that the best upper and lower bounds listed in tables 4.4(i), 4.4(ii), 4.6(i), 4.6(ii), and 4.6(iii) do not improve on the trivial bounds mentioned above: $2/3 \le \mu(q) \le 1$. Even the more recent results from \cite{ElGi} don't close the gap significantly; it seems that there is more work to be done before we have a good understanding of maximal caps in higher dimensions. It is also important to mention that the upper bounds from \cite{ElGi} are only competitive with the trivial $m_2(n,q) < O(q^{n-1})$ when $n$ is much larger than $\log(q)$.

In this paper, we will be looking at a generalization of caps. Rather than just restricting the number of points on lines, we can restrict the number of points on $k$-dimensional affine subspaces of $AG(n,q)$, known as $k$-flats.

\begin{definition}
Suppose $3 \le m \le n+2$ and $A \subset AG(n,q)$ has size greater than $m$. Then $A$ is \textbf{\boldmath $m$-general} if no $m$ points of $A$ lie on a single $(m-2)$-flat. Equivalently, $A$ is $m$-general if any $m$-point subset of $A$ is in general position.
\end{definition}

\noindent
Note that a cap is the same as a $3$-general set. In the language of \cite{Hirsch}, an $m$-general set $A$ is essentially the same as an $(|A|,m-1)$-set, though by our definition, any $m$-general set is also $k$-general for $3 \le k\le m$. If the maximum size of an $m$-general set is $M_{m-1}(n,q)$ (the notation used in \cite{Hirsch}), let
$$
\mu_m(q) = \limsup_{n \to \infty} \frac{\log_q\left(M_{m-1}(n,q)\right)}{n}.
$$
Trivially, we have $\frac{1}{m-1} \le \mu_m(q) \le 1$. The lower bound is due to the following observation:
Suppose $A$ is an $m$-general set in $AG(n,q)$. Then there are precisely $\binom{|A|}{m-1}$ distinct $(m-2)$-flats each containing $m-1$ points of $A$. The union of these flats covers at most $q^{m-2}\binom{|A|}{m-1}$ points of $AG(n,q)$. So as long as $q^{m-2}\binom{|A|}{m-1} < q^n$, there are other points that can be added to $A$ to create a larger $m$-general set. Solving for $|A|$ gives the result.

A recent paper of Huang, Tait, and Won (\cite{HTW}) uses a clever combinatorial argument to prove that $\mu_4(3) = 0.5$. For $m \ge 4$, an $m$-general set must be $4$-general, and thus an immediate consequence is that $\mu_m(3) \le 0.5$. The upper bounds we recover for $\mu_m(3)$ in this paper are not as good as 0.5 when $4 \le m \le 9$ (see Table \ref{tableb}).

\begin{theorem}\label{main}
Let $n$ be a positive integer, $q$ a power of a prime $p$, and $m$ an integer such that $3 \leq m \leq n+2$. Suppose also that $q$ is odd, or $m$ and $q$ are both even. Then there is an $\alpha$ between $0$ and $1$ so that
$$
M_{m-1}(n,q) <  2m + m \cdot \min_{t \in (0,1)}\left(t^{-\frac{q-1}{m}} \cdot \frac{1-t^q}{1-t}\right)^n \le 2m + m \cdot\left(\frac{me^{1-\frac{\alpha}{m}}}{m^2-\alpha m+\alpha}q + C \right)^n,
$$
where $C$ depends only on $m$.
\end{theorem}

In particular, this tells us that
\begin{equation}\label{mubound}
\mu_m(q) \le 1 - \log_q\left(\frac{m^2-\alpha m+\alpha}{me^{1-\frac{\alpha}{m}}}\right)+O\big ((q\log q)^{-1}\big).
\end{equation}
The value of $\alpha$, which we will show how to estimate in lemma \ref{asymptotic}, depends on $m$.

The restriction $m \leq n+2$ makes sense, as in the space $AG(n,q)$, it is not possible to have $n+2$ points in general position. On the other hand, the omission of the case where $q$ is even and $m$ is odd is not founded on any geometric principles; it is merely an artifact of the methodology we will see here. It is very possible that a similar result holds for this case using a slightly different approach.

\begin{remark} In \cite{ElGi}, the authors obtain upper bounds on the sizes of subsets of $AG(n,q)$ with no \textit{arithmetic progressions} along lines. In this paper, in the $m=3$ case, we simply require that no three points be on a line. Despite this stricter imposition, the bounds we achieve in the $m=3$ case do not improve over those in \cite{ElGi}. It may be possible to refine the methods in this paper to improve the bound, but it would require a more careful investigation of the polynomial $G$ (see equation \eqref{g}).
\end{remark}

%WWWWWWWWWWWWWWWWWWWWWWWWWWWWWWWWWWWWWWWWWWWWWWWWWWWWWWWWWWWWWWWWWWWWWWWWWWWWWWWWWWWWWWWWWWWWW
%WWWWWWWWWWWWWWWWWWWWWWWWWWWWWWWWWWWWWWWWWWWWWWWWWWWWWWWWWWWWWWWWWWWWWWWWWWWWWWWWWWWWWWWWWWWWW

\section{Rank of a function}
Our result relies heavily on the method of Croot, Lev, Pach, Ellenberg, and Gijswijt as outlined by Tao in \cite{Tao}. Tao introduces the ``rank" of a function, which has a close connection with matrix rank:

\begin{definition} The function $F:A^k \to X$ is said to have \textbf{rank} $r$ ($\rank(F) = r$) if $r$ is the smallest integer that allows us to write
$$
F(x_1,x_2, \ldots, x_k) = \sum_{n = 1}^r f_n(x_{m_n})g_n(x_1, \ldots, x_{m_n-1}, x_{m_n+1},\ldots, x_k).
$$
for some $m_n \in \{1, 2, \ldots, k\}$ and functions
$$
f_n: A \to X \quad\quad\quad g_n:A^{k-1} \to X.
$$
\end{definition} For instance, if $F:\bb{R}^2 \to \bb{R}$, where
$$
F(x,y) = x^2y + xy^2+ 2x+y^2 + y + 2,
$$
then $F$ has rank 2, since $F(x,y)$ can be written as $(x^2+1)y + (x+1)(y^2+2)$ but cannot be written in the form $f(x)g(y)$. We will occasionally abuse notation and write, for instance,\\ ``$\rank(x^2y + xy^2+ 2x+y^2 + y + 2) = 2$" when we mean ``$\rank(F) = 2$."

It is important to note here that the rank of a function depends on the number of variables $F$ takes. If $F$ is a function of $k$ variables, but only $k-1$ of them appear in the definition of $F$, then the rank of $F$ is 1 (or 0 if $F$ is identically 0). For instance,
$$
F(x,y,z) = x^2y + xy^2+ 2x+y^2 + y + 2
$$
is a rank 1 function, since $F(x,y,z) = f(z)g(x,y)$, where $f(z) = 1$ and $g(x,y) = F(x,y,0)$. When clarity is needed, we will say that the \boldmath ${k}$-\textbf{rank} \unboldmath  of $F$ is $r$ ($\rank_k(F) = r$) to stress that its rank, as a function of $k$ variables, is $r$.

Before looking at some properties of rank, we introduce a useful bit of notation:
\begin{definition}\label{vrow}
Let $A = \{a_1, \ldots, a_{|A|}\}$ be a finite set and $f$ a function on $A$. We define \boldmath $v_{\text{row}}(f),v_{\text{col}}(f)$ \unboldmath to be the $|A|$-dimensional row and column vectors with $f(a_i)$ in the $i^\text{th}$ position.
\end{definition}

\pagebreak

\begin{proposition}

Let $A = \{a_1, \ldots, a_{|A|}\}$ be a finite set, $X$ a field, and $\mathcal{F}_k$ the vector space over $X$ of $k$-variable functions $f:A^k \to X$. Let $F,G \in \mathcal{F}_k$. Then the following properties hold:

\begin{Properties}[label=\textbf{R.\arabic*}]

\item \label{R1} $\ds \rank_k(F+G) \leq \rank_k(F) + \rank_k(G)$.

\item \label{R2} If $B \subset A$, then $\rank_k \left(F\big|_B \right) \leq \rank_k(F)$.

\item \label{R3} $\rank_k(F) \leq |A|$.

\item \label{R4} If $H \in \mathcal{F}_2$ and $M$ is the $|A| \times |A|$ matrix with $m_{ij} = H(a_i,a_j)$, then $\rank_2(H) \geq \rank(M)$.

\end{Properties}
For properties \ref{R5}, \ref{R6}, \ref{R7}, $f_n \in \mathcal{F}_1$, $g_n \in \mathcal{F}_k$, and the function $h\in \mathcal{F}_{k+1}$ is defined by $$h(x,y_1,y_2 , \ldots, y_{k}) = \sum_{n = 1}^r f_n(x)g_n(y_1,\ldots,y_k).$$
\begin{Properties}[resume]

\item \label{R5} If $\{\tilde{f}_n: 1 \leq n \le r\} \subset \mathcal{F}_1$ so that $\{f_n: 1 \leq n \leq r\} \subset \SPAN\{\tilde{f}_n: 1 \leq n \leq \tilde{r}\}$, 
then there exists $\{\tilde{g}_n: 1 \le n \le \tilde{r}\} \subset \mathcal{F}_k$ so that
$$
h(x,y_1, \ldots, y_k) = \sum_{n = 1}^{\tilde{r}} \tilde{f}_n(x)\tilde{g}_n(y_1,\ldots,y_k).
$$

\item \label{R6}  Let $M$ be the $|A| \times r$ matrix whose columns are $v_\text{col}(f_n)$. Then $
\rank_{k+1}(h) \leq \rank(M)$.

\item \label{R7} If $\ds \rank_{k+1}(h) = r$,
then the $f_n$ are linearly independent in $\mathcal{F}_1$.

\end{Properties}

\end{proposition}

\begin{proof}

Properties \ref{R1} and \ref{R2} are trivial.

\ref{R3}: Let $\delta_a$ be the function on $A$ which is $1$ at $a$ and $0$ otherwise. Then
$$
F(x_1,x_2, \ldots, x_k) = \sum_{a \in A} \delta_a(x_1)F(a,x_2,x_3,\ldots, x_k).
$$

\ref{R4}: Suppose $H$ has rank $r$. Then $\ds H(x,y) = \sum_{n = 1}^{r} f_n(x)g_n(y)$
for functions $f_n,g_n: A \to X$. For each $n$, let $M_n$ be the $|A| \times |A|$ matrix $v_\text{col}(f_n) v_\text{row}(g_n)$. Since each $M_n$ has rank at most 1, $M = \sum_{n = 1}^r M_n$ is a matrix of rank at most $r$.

\vspace{.1 in}
\ref{R5}: For each fixed choice of $(y_1,\ldots, y_k) \in A^k$, elementary linear algebra tells us there are elements  $s_n(y_1,\ldots,y_{k}) \in X$ for $1 \leq n \leq \tilde{r}$
so that
$$
\sum_{n = 1}^r g_n(y_1,\ldots,y_k)v_\text{col}(f_n) = \sum_{n = 1}^{\tilde{r}}s_n(y_1,\ldots,y_{k}) v_\text{col}(\tilde{f}_n).
$$
Thus we may simply define the functions $\tilde{g}_n$ by $\tilde{g}_n(y_1, \ldots, y_k) = s_n(y_1,\ldots,y_{k})$.

Properties \ref{R6} and \ref{R7} follow immediately from $\ref{R5}$.

\end{proof}

%WWWWWWWWWWWWWWWWWWWWWWWWWWWWWWWWWWWWWWWWWWWWWWWWWWWWWWWWWWWWWWWWWWWWWWWWWWWWWWWWWWWWWWWWWWWWW
%WWWWWWWWWWWWWWWWWWWWWWWWWWWWWWWWWWWWWWWWWWWWWWWWWWWWWWWWWWWWWWWWWWWWWWWWWWWWWWWWWWWWWWWWWWWWW

\section{Setup for the proof of theorem \ref{main}}

Fix integers $n$ and $m$ with $3 \le m \le n+2$. For any set $S \subset AG(n,q)$, define $G_m^S: S^m \to \bb{F}_q$ by
\begin{equation}\label{g}
G_{m}^S(x_1, \ldots, x_{m}) = \sum_{t_1, \dots, t_{m-1} \in \bb{F}_q} \prod_{j = 1}^n \left[1 - \left(\sum_{i = 1}^{m-1} t_i(x_{ij}-x_{mj}) \right)^{q-1}\right].
\end{equation}
where $x_{ij}$ is the $j^\text{th}$ coordinate of point $x_i$.

\vspace{.1 in}
\noindent
Notice that the bracketed expression is equal to $1$ if
$$
\begin{bmatrix} t_1 \\ t_2 \\ \vdots \\ t_{m-1} \end{bmatrix} \boldsymbol{\cdot} \begin{bmatrix} x_{1j} -x_{mj}\\ x_{2j}-x_{mj} \\ \vdots \\x_{(m-1)j}-x_{mj} \end{bmatrix} = 0
$$
and equal to $0$ otherwise. Thus $G_{m}^S(x_1, \ldots, x_{m})$ is equal to the number of elements, modulo $p$, in
$$
\text{null}\Big(\{x_i-x_{m}: 1 \le i \le m-1\}\Big).
$$
Since the size of a vector space over $\bb{F}_q$ must be a power of $q$, we see that $G_{m}^S(x_1, \ldots, x_{m})$ evaluates to $1$ if the vectors of $\{x_i - x_{m}: 1 \le i \le m-1\}$ are linearly independent, and $0$ otherwise.

Now suppose that the set $A$ is $k$-general. If $x_1, \ldots, x_{k}$ are points of $A$, then \\$\{x_i-x_{k}: 1 \le i \le k-1\}$ is a set of $k-1$ linearly independent vectors if and only if $x_1, \ldots, x_{k}$ are all distinct. Therefore, if we define a function $T_m^S: S^m \to \bb{F}_q$ for any set $S \subset AG(n,q)$ by
\begin{equation}\label{t}
T_{m}^S(x_1, \ldots, x_{m}) = \begin{cases} 1 & \text{all } x_j \text{ are distinct}\\
0 &	\text{ otherwise}
\end{cases},
\end{equation}
then $T_m^S = G_m^S$ when $S$ is $m$-general.

From here, the general idea is to follow the procedure of \cite{Tao}. We will divide our argument into three lemmas:

\begin{lemma}\label{matrices}
Let $S \subset AG(n,q)$ and $m \ge 2$. If $q$ is odd or $q$ and $m$ are both even, then
$$
\rank(T_m^S) \ge |S|-2m+3.
$$
\end{lemma}

\begin{lemma}\label{polynomial}
For any set $S \subset AG(n,q)$ and $m \ge 3$,
$$
\rank(G_m^S) \le m \cdot\min_{t \in (0,1)}\left(t^{-\frac{q-1}{m}} \cdot \frac{1-t^q}{1-t}\right)^n.
$$
\end{lemma}

\pagebreak

\begin{lemma}\label{asymptotic}
Fix an integer $m \geq 3$ and let
$$
h_{q}(x) = x^{-\frac{q-1}{m}} \cdot \frac{1-x^q}{1-x}.
$$
Then on $(0,1)$, $h_q(x)$ attains a minimum value of
$$
\frac{me^{1-\frac{\alpha}{m}}}{m^2-m\alpha+\alpha}q + O(1),
$$
where $\alpha$ is the unique value in $(0,1)$ satisfying $$ \alpha = \frac{m^2-m\alpha+\alpha}{e^{m - \alpha}}.$$
\end{lemma}

\vspace{.2 in}
When $A \subset AG(n,q)$ is $m$-general (and $q$ is even or $m$ is odd), combining lemmas \ref{matrices} and \ref{polynomial} gives us
$$
|A|-2m+3 \leq \rank(T_m^A) = \rank(G_m^A) \leq m \cdot\min_{t \in (0,1)}\left(t^{-\frac{q-1}{m}} \cdot \frac{1-t^q}{1-t}\right)^n,
$$
and therefore
$$
M_{m-1}(n,q) \le 2m + m \cdot \min_{t \in (0,1)}\left(t^{-\frac{q-1}{m}} \cdot \frac{1-t^q}{1-t}\right)^n.
$$
In lemma \ref{asymptotic}, we verify that $\ds \min_{t \in (0,1)}\left(t^{-\frac{q-1}{m}} \cdot \frac{1-t^q}{1-t}\right)$ is well-defined and bounded above by
$$
\frac{me^{1-\frac{\alpha}{m}}}{m^2-\alpha m+\alpha}q + O(1),
$$
completing the proof of theorem \ref{main}.

\vspace{.1 in}
\begin{remark} In lemma \ref{matrices}, we see that the rank of $T_m^S$ is typically around $|S|$, but this surprisingly does not hold when $p=2$ and $m$ is odd, hence the omission of that case. Indeed, in characteristic $2$ it is easy to verify that
$$
T_{2k+1}^S(x_1, x_2, \ldots, x_{2k+1}) = \sum_{i = 1}^{2k+1} T_{2k}^S(x_1, \ldots, x_{i-1},x_{i+1}, \ldots, x_{2k+1})
$$
and thus $\rank(T_{2k+1}^S) \le 2k+1$.
\end{remark}
\vspace{.1 in}

%WWWWWWWWWWWWWWWWWWWWWWWWWWWWWWWWWWWWWWWWWWWWWWWWWWWWWWWWWWWWWWWWWWWWWWWWWWWWWWWWWWWWWWWWWWWWW
%WWWWWWWWWWWWWWWWWWWWWWWWWWWWWWWWWWWWWWWWWWWWWWWWWWWWWWWWWWWWWWWWWWWWWWWWWWWWWWWWWWWWWWWWWWWWW

\section{Proof of lemma \ref{matrices}}

We proceed by induction on $m$ and begin with the case $m = 2$. Enumerate $B = \{b_1, \ldots, b_{|B|}\}$ and let $M$ be the matrix with $ m_{ij} = T_2^B(b_i,b_j)$. By the definition of $T_2^B$, $M$ is the matrix which has zeros along the diagonal and ones everywhere else. Thus $M$ has rank at least $|B|-1$, and by claim \ref{R4}, $\rank(T_2^B) \geq |B|-1$. (Note: The matrix $M$ has rank $|B|$ unless $|B| \equiv 1 \bmod p$, when the rank is $|B|-1$.)

\vspace{.2 in}

We will first consider the case where $q$ is odd. Fix an integer $k \ge 2$ and assume that for any $S \subset AG(n,q)$, $\rank_j(T_j^S) \ge |S|-2j+3$ when $2 \le j \le k$. Fix $B \subset AG(n,q)$ and let $r$ be the $(k+1)$-rank of $T_{k+1}^B$. Then there are functions $f_{i,\alpha}: B \to \bb{F}_q$, $g_{i,\alpha}: B^{k} \to \bb{F}_q$ so that
\begin{equation}\label{rankeq}
T_{k+1}^B(x_1, \ldots, x_{k+1}) = \sum_{i = 1}^{k+1} \sum_{\alpha \in I_i} f_{i,\alpha}(x_i)g_{i,\alpha}(x_1, \ldots, x_{i-1},x_{i+1}, \ldots, x_{k+1})
\end{equation}
where the indexing sets $I_i$ are disjoint and $\sum_{i=1}^{k+1} |I_i| = r$. Let $\1_B:B \to \bb{F}_q$ be the function which is identically $1$ on $B$.

In most situations, we can prove that $\rank(T_{k+1}^S) \ge |S|-2(k+1)+3$ by assuming that $\rank(T_{k}^S) \ge |S|-2k+3$. This is demonstrated in Case 1. However, we run into a hiccup when $k$ happens to be divisible by $p$. To get around this, we instead appeal to the assumption that $\rank(T_{k-1}^S) \ge |S|-2(k-1)+3$, which we take care of in Case 2.

\vspace{.2 in}

\noindent
\textbf{\underline{Case 1}}: $p \centernot\mid k$, or $ p \mid k$ and $\ds \1_B \notin \bigcap_{i=1}^{k+1}\SPAN\Big(\{f_{i,\alpha}: \alpha \in I_{i} \}\Big)$

If $p \centernot\mid k$, let
$$
H = \SPAN\Big(\{\1_B\} \cup \{f_{k+1,\alpha}: \alpha \in I_{k+1}\}  \Big).
$$
Otherwise, since $T_{k+1}^B$ is symmetric in all variables, we may assume without loss of generality that\\ $\1_B \notin \SPAN\Big(\{f_{k+1,\alpha}: \alpha \in I_{k+1} \}\Big)$ and let
$$
H = \SPAN\Big(\{f_{k+1,\alpha}: \alpha \in I_{k+1}\}  \Big).
$$
In either case, let $H^\perp$ be the orthogonal complement of $H$ with respect to the usual inner product.

\vspace{.1 in}
Because the dimension of $H$ is at most $|I_{k+1}|+1$, the dimension $d$ of $H^\perp$ is at least $|B|-|I_{k+1}|-1$. Find a set $B' \subset B$ and an appropriate basis $\mathcal{U} = \{h_{1}, h_{2}, \ldots, h_{d}\}$ for $H^\perp$ so that $|B'| = d$ and
$$
\begin{bmatrix} v_\text{col}\left(h_{1}\big|_{B'}\right) & v_\text{col}\left(h_{2}\big|_{B'}\right) & \cdots & v_\text{col}\left(h_{d}\big|_{B'}\right)\\[3pt] \end{bmatrix} = 
\begin{bmatrix} 0 & 1 & 1 & \cdots & 1 & 1 \\
								1 & 0 & 1 & \cdots & 1 & 1 \\
								1 & 1 & 0 & \cdots & 1 & 1 \\
								\vdots & \vdots & \vdots & \ddots & \vdots & \vdots\\
								1 & 1 & 1 & \cdots & 0 & 1 \\
								1 & 1 & 1 & \cdots & 1 & 1 \\ 
\end{bmatrix}
$$
(see definition \ref{vrow}). If $p \centernot\mid k$, we simply let $\bar{h} = h_d$. Otherwise, since $\1_B \notin H$, there must be a function $\bar{h} \in \mathcal{U}$ so that $\bar{h}$ is not orthogonal to $\1_B$, i.e. $\sum_{b \in B}\bar{h}(b) \neq 0$.

\vspace{.1 in}

Multiplying both sides of \eqref{rankeq} by $\bar{h}(x_{k+1})$ and summing over $x_{k+1} \in B$, the right side becomes
\begin{equation}\label{right}
\sum_{i = 1}^{k} \sum_{\alpha \in I_i} \left(f_{i,\alpha}(x_i) \sum_{x_{k+1} \in B} \bar{h}(x_{k+1})g_{i,\alpha}(x_1, \ldots, x_{i-1},x_{i+1}, \ldots, x_{k+1})\right),
\end{equation}
which has rank at most $r -|I_{k+1}|$.

\vspace{.1 in}
\noindent
On the left side we get
\begin{align}
&\sum_{x_{k+1}\in B}\bar{h}(x_{k+1})T_{k+1}^B(x_1, \ldots, x_{k+1}) \nonumber\\
= \quad&T_{k}^B(x_1, \ldots, x_k) \sum_{\mathclap{{x \in B \backslash\{x_1, \ldots, x_k\}}}}\bar{h}(x) \nonumber\\
= \quad &T_{k}^B(x_1, \ldots, x_{k})\left(\sum_{x \in B} \bar{h}(x) - \sum_{i=1}^k \bar{h}(x_i)\right).\label{left}
\end{align}

Let $B''= \{b \in B: \bar{h}(b) = 1\}$ and notice that $|B''| \ge d-1$. Restrict the domain of both \eqref{right} and \eqref{left} to $(B'')^k$. By \ref{R2}, the rank of \eqref{right} is still no more than $r -|I_{k+1}|$. Note that the second sum in \eqref{left} simplifies to $k$ since $\bar{h}\big|_{B''} \equiv 1$. If $p\centernot\mid k$, then the first sum is $0$ since $\1_B \in H$. If $p \mid k$, then the first sum is some nonzero constant by our construction of $\bar{h}$. In either case, we are left with $cT_{k}^{B''}(x_1, \ldots, x_{k})$ for some $c \neq 0$, and $\rank_k(cT_k^{B''}) \geq |B''|-2k+3$ by the inductive hypothesis. Comparing the ranks of \eqref{right} and \eqref{left}, we see
$$
r - |I_{k+1}| \geq |B''|-2k+3  \geq |B|-|I_{k+1}|-2k+1
$$
and thus
$$
\rank(T_{k+1}^B) = r \geq |B|-2(k+1)+3.
$$
\noindent
\textbf{\underline{Case 2}}: $p \mid k$ and $\ds \1_B \in \bigcap_{i=1}^{k+1}\SPAN\Big(\{f_{i,\alpha}: \alpha \in I_{i} \}\Big)$.

Notice that $k+1 \geq 4$, $T_{k+1}^B$ is symmetric, and $r \le |B|$ by \ref{R3}. Therefore, we may assume without loss of generality that $|I_k| + |I_{k+1}|  < |B|$. For $i = k, k+1$, let
$$
H_{i} = \SPAN\Big(\{f_{i,\alpha}: \alpha \in I_{i}\}  \Big)
$$
and let $H_i^\perp$ be the orthogonal complement.

\vspace{.1 in}
Because the dimension of $H_i$ is $|I_i|$ (by \ref{R7}), the dimension $d_i$ of $H_i^\perp$ is $|B|-|I_{i}|$. Find a set $B_i \subset B$ and an appropriate basis $\mathcal{U}_i = \{h_{i,1}, h_{i,2}, \ldots, h_{i,d_i}\}$ for $H_i^\perp$ so that $|B_i| = d_i$ and
$$
\begin{bmatrix} v_\text{col}\left(h_{i,1}\big|_{B_i}\right) & v_\text{col}\left(h_{i,2}\big|_{B_i}\right) & \cdots & v_\text{col}\left(h_{i,d_i}\big|_{B_i}\right)\\[3pt] \end{bmatrix} = 
\begin{bmatrix} 0 & 1 & 1 & \cdots & 1 & 1 \\
								1 & 0 & 1 & \cdots & 1 & 1 \\
								1 & 1 & 0 & \cdots & 1 & 1 \\
								\vdots & \vdots & \vdots & \ddots & \vdots & \vdots\\
								1 & 1 & 1 & \cdots & 0 & 1 \\
								1 & 1 & 1 & \cdots & 1 & 1 \\ 
\end{bmatrix}
$$
%(Note that this requires a specific enumeration of the elements of $B_i$ which depends on $i$.)
Since $d_k + d_{k+1} = 2|B| - |I_k| - |I_{k+1}| > |B|$, there must be some $\bar{h}_k \in \mathcal{U}_k$ and $\bar{h}_{k+1} \in \mathcal{U}_{k+1}$ so that $\bar{h}_k$ and $\bar{h}_{k+1}$ are non-orthogonal.

Multiplying both sides of \eqref{rankeq} by $\bar{h}_k(x_{k})\bar{h}_{k+1}(x_{k+1})$ and summing over all $x_k,x_{k+1} \in B$, the right side becomes
\begin{equation}\label{right2}
\sum_{i = 1}^{k-1} \sum_{\alpha \in I_i} \left(f_{i,\alpha}(x_i) \sum_{x_k,x_{k+1} \in B} \bar{h}_k(x_k)\bar{h}_{k+1}(x_{k+1})g_{i,\alpha}(x_1, \ldots, x_{i-1},x_{i+1}, \ldots, x_{k+1})\right),
\end{equation}
which has rank at most $r -|I_{k+1}|-|I_k|$.

Meanwhile, the left side simplifies to
$$
T_{k-1}^B(x_1, \ldots, x_{k-1}) \sum_{\mathclap{\substack{x,y \in B \backslash\{x_1, \ldots, x_{k-1}\}\\ x \neq y}}}\bar{h}_k(x)\bar{h}_{k+1}(y).
$$
Abbreviating $\{x_1, x_2, \ldots, x_{k-1}\}$ as $\mathcal{X}$ and expanding,
\begin{align*}
T_{k-1}^B(x_1, \ldots, x_{k-1})&\left(\sum_{x, y \in B \backslash \mathcal{X}} \bar{h}_k(x)\bar{h}_{k+1}(y) - \sum_{x \in B \backslash \mathcal{X}} \bar{h}_k(x)\bar{h}_{k+1}(x)\right)\\
= T_{k-1}^B(x_1, \ldots, x_{k-1}) &\Bigg(\sum_{x \in B \backslash \mathcal{X}} \bar{h}_k(x)\sum_{y \in B}\bar{h}_{k+1}(y)- \sum_{y \in \mathcal{X}} \bar{h}_{k+1}(y)\sum_{x \in B}\bar{h}_{k}(x)\\
& + \sum_{x, y \in  \mathcal{X}} \bar{h}_k(x)\bar{h}_{k+1}(y) - \sum_{x \in B \backslash \mathcal{X}} \bar{h}_k(x)\bar{h}_{k+1}(x)\Bigg).
\end{align*}
Since $\1_B \in H_k \cap H_{k+1}$, the first two terms disappear, leaving
$$
T_{k-1}^B(x_1, \ldots, x_{k-1})\left[ \left(\sum_{x \in \mathcal{X}} \bar{h}_k(x)\right)\left(\sum_{y \in \mathcal{X}}\bar{h}_{k+1}(y)\right) - \sum_{x \in B} \bar{h}_k(x)\bar{h}_{k+1}(x) + \sum_{i = 1}^{k-1} \bar{h}_k(x_i)\bar{h}_{k+1}(x_i)\right].
$$
Since $\bar{h}_k$ and $\bar{h}_{k+1}$ are not orthogonal, we have
\begin{equation}\label{left2}
T_{k-1}^B(x_1, \ldots, x_{k-1})\left[ \left(\sum_{i = 1}^{k-1} \bar{h}_k(x_i)\right)\left(\sum_{i=1}^{k-1}\bar{h}_{k+1}(x_i)\right)-c + \sum_{i = 1}^{k-1} \bar{h}_k(x_i)\bar{h}_{k+1}(x_i)\right] 
\end{equation}
for some $c \neq 0$.

Let $B' = \{x \in B: \bar{h}_{k+1}(x) = \bar{h}_k(x) = 1\}$. By our constructions of $\bar{h}_{k+1}$ and $\bar{h}_k$,
$$
|B'| \geq d_k-1 + d_{k+1}-1 - |B| = |B| - |I_k| - |I_{k+1}| - 2.
$$
Restrict the domains of both \eqref{left2} and \eqref{right2} to $(B')^k$. Since $p \mid k$, expression \eqref{left2} simplifies to
$$
T_{k-1}^{B'}(x_1, \ldots, x_{k-1})\big((k-1)^2 -c+ (k-1)\big) = -cT_{k-1}^{B'}(x_1, \ldots, x_{k-1}),
$$
a function whose $(k-1)$-rank is at least $|B'|-2k+5$ by our inductive hypothesis. Comparing the ranks of \eqref{right2} and \eqref{left2}, we see
$$
r - |I_{k+1}|-|I_k| \geq |B'|-2k+5  \geq |B|-|I_{k}|- |I_{k+1}|-2k+3
$$
and thus
$$
\rank(T_{k+1}^B) = r \geq |B|-2(k+1)+5\geq |B|-2(k+1)+3.
$$
This completes the induction for odd $q$.

\pagebreak

The case of even $q$ is very similar to Case 2, as the induction will again take steps of size 2, iterating only through even values of $m$. However, we will need to take some extra steps to ensure that $H_i$ contains $\1_B$, which was already assumed in Case 2. (Recall that we will not be obtaining a result for even $q$ when $m$ is odd.)

Let $k\geq 3$ be odd and assume that $\rank(T_{k-1}^S) \geq |S|-2(k-1)+3$. Notice that when $|B| \le 5$, the desired result
$$
\rank(T_{k+1}^B) \geq |B|-2(k+1)+3
$$
is trivial, and therefore we may assume $|B| > 5$. Using $|B| > 5$, $k+1 \ge 4$ and the fact that $T_{k+1}^B$ is symmetric, we may assume that $|I_k| + |I_{k+1}| \le |B|/2 < |B|-2$.

Let $H_i = \SPAN\Big(\{\1_B\} \cup \{f_{i,\alpha}: \alpha \in I_{i}\}  \Big)$. This time, we only know that the dimension $d_i$ of $H_i^\perp$ is at least $|B|-|I_i|-1$, but we still have
$$
d_k + d_{k+1} \ge 2|B| - |I_k| - |I_{k+1}| -2> |B|.
$$
We construct $B_i$, $\mathcal{U}_i$, $\bar{h}_i$, and $B'$ as before. Again, we multiply both sides of \eqref{rankeq} by $\bar{h}_k(x_{k})\bar{h}_{k+1}(x_{k+1})$, sum over all $x_k,x_{k+1} \in B$, and restrict to $(B')^k$ to get
\begin{align*}
cT_{k-1}^{B'}(x_1, \ldots,x_{k-1})& = \\
\sum_{i = 1}^{k-1} &\sum_{\alpha \in I_i} \left(f_{i,\alpha}(x_i) \sum_{x_k,x_{k+1} \in B} \bar{h}_k(x_k)\bar{h}_{k+1}(x_{k+1})g_{i,\alpha}(x_1, \ldots, x_{i-1},x_{i+1}, \ldots, x_{k+1})\right) 
\end{align*}
for some $c \neq 0$. However, in this case,
$$
|B'| \geq d_k-1 + d_{k+1}-1 - |B| \ge |B| - |I_k| - |I_{k+1}| - 4.
$$
Nevertheless, comparing the ranks of both sides of the equation still yields
$$
\rank(T_{k+1}^B) = r \geq  |B|-2(k+1)+3,
$$
completing the induction.

%WWWWWWWWWWWWWWWWWWWWWWWWWWWWWWWWWWWWWWWWWWWWWWWWWWWWWWWWWWWWWWWWWWWWWWWWWWWWWWWWWWWWWWWWWWWWW
%WWWWWWWWWWWWWWWWWWWWWWWWWWWWWWWWWWWWWWWWWWWWWWWWWWWWWWWWWWWWWWWWWWWWWWWWWWWWWWWWWWWWWWWWWWWWW

\section{Proof of lemma \ref{polynomial}}

This proof uses the same general procedure that can be found in \cite{ElGi}. However, we are dealing with a more general case and will present all of the necessary details here. Looking back at equation \eqref{g}, we see $G_m^S$ is a polynomial in $mn$ $\bb{F}_q$-valued variables $x_{ij}$. Let $P$ be the set of monomials appearing in the expansion of $G_m^S$. Each monomial $\rho \in P$ can be written as
$$
\scalebox{1.2}{$\displaystyle\rho(x_1, \ldots, x_m) = c\prod_{i =1}^m\prod_{j=1}^n x_{ij}^{e_{ij}}$},
$$
where the coefficient $c \in \bb{F}_q$ and the $e_{ij} \in \bb{N}$ depend on $\rho$. (For convenience, we will consider $0$ an element of $\bb{N}$.)

By \eqref{g}, each $e_{ij}$ is no greater than $q-1$ and
$$
\sum_{i = 1}^m \sum_{j=1}^n e_{ij} \leq (q-1)n.
$$
Thus, there must be some index $i$ for which $\displaystyle \sum_{j=1}^n e_{ij} \leq \frac{(q-1)n}m$.  For each $\rho \in P$, choose such an index and call it $\kappa = \kappa(\rho)$. We then separate out the $\kappa$-factors of $\rho$:
$$
\scalebox{1.2}{$\displaystyle\rho(x_1, \ldots, x_m) = c\prod_{j =1}^n x_{\kappa j}^{e_{\kappa j}}\prod_{i \neq \kappa}\prod_{j=1}^n x_{ij}^{e_{ij}}$}.
$$
Letting $\displaystyle f_\rho(x_{\kappa}) = \prod_{j =1}^n x_{\kappa j}^{e_{\kappa j}}$ and $\displaystyle g_\rho(x_1, \ldots, x_{\kappa -1}, x_{\kappa +1}, \ldots, x_m) = c\prod_{i \neq \kappa}\prod_{j=1}^n x_{ij}^{e_{ij}}$, we have
$$
G_m^S(x_1, \ldots, x_m) = \sum_{i=1}^m \sum_{\substack{\rho \in P \\ \kappa(\rho) = i}} f_\rho(x_i)g_\rho(x_1, \ldots, x_{i -1}, x_{i +1}, \ldots, x_m).
$$
Next, group together the polynomials with matching ``$\kappa$-factors," i.e. for $e = (e_1,e_2, \ldots, e_n) \in \bb{N}^n$,
$$
M_i(e) = \left\{\rho \in P: \kappa(\rho) = i,  f_\rho(x_i) = \prod_{j=1}^n x_{ij}^{e_j}\right\}.
$$
We then reorganize the sum:
$$
G_m^S(x_1, \ldots, x_m) = \sum_{i=1}^m \sum_{e\in \bb{N}^n} \left[\left(\prod_{j = 1}^n x_{ij}^{e_j}\right) \sum_{\rho \in M_i} g_\rho(x_1, \ldots, x_{i -1}, x_{i +1}, \ldots, x_m)\right].
$$
Notice that the expression in square brackets is a function of rank $1$. Therefore, by \ref{R1}, the rank of $G_m^S$ is bounded above by
$$
m\cdot \max_{1 \leq i \leq m}\#\{e \in \bb{N}^n: M_i(e) \neq \emptyset\}.
$$
As we observed earlier, $M_i(e)$ is empty unless  $e_j \leq q-1$ for all $j$ and $\sum_{j = 1}^n e_j \leq \frac{(q-1)n}m$. Thus the rank of $G_m^S$ is bounded above by the number of $n$-tuples in $\bb{N}^n$ in which each coordinate is no greater than $q-1$ and the sum of the coordinates is no greater than $ \frac{(q-1)n}m$.

%This is the key step. Note that

%Looking back at equation \ref{g}, we see $G_m^S$ is a polynomial in $mn$ $\bb{F}_q$-valued variables $x_{ij}$. Notice that each of the monomial terms in the expansion of $G_m^S$ has total degree no greater than $(q-1)n$. For each monomial $\rho$, let $i_\rho$ be such that the total degree of $x_{i_\rho}$ is smallest. By the pigeonhole principle, the total degree of $x_{i_\rho}$ is at most $\frac{(q-1)n}{m}$.

\vspace{.1 in}

For $\alpha, \beta, \gamma \in \bb{N}$, let $\Lambda \left(\alpha,\beta,\gamma \right)$ be the number of $\alpha$-tuples of elements in $\{0,1,2, \ldots, \beta\}$ with sum no greater than $\gamma$. It is easy to verify that the number of $\alpha$-tuples with sum \textit{equal} to $i$ is $ \ds [x^i]\left(\frac{1-x^{\beta+1}}{1-x}\right)^\alpha$
and therefore
$$
\Lambda \left(\alpha,\beta,\gamma \right) = \sum_{i = 0}^{\gamma}[x^i]\left(\frac{1-x^{\beta+1}}{1-x}\right)^\alpha.
$$
We can derive a slight variation on the familiar saddle point bound: suppose that $f(x) = \sum_{i = 0}^{\infty} c_ix^i$ on $(0,1)$ and each $c_i$ is a non-negative real. Then for any non-negative integer $N$ and any $t \in (0,1)$, we have
$$
\sum_{i=0}^{N} [x^i]f(x) = \sum_{i = 0}^{N} c_i \leq \sum_{i = 0}^{\infty} c_i t^{i-N} = t^{-N}f(t).
$$
Therefore
$$
\Lambda \left(\alpha,\beta,\gamma \right) \le t^{-\gamma} \left( \frac{1-t^{\beta+1}}{1-t}\right)^\alpha
$$
for all $t \in (0,1)$. Applying this to the problem at hand,
\begin{align*}
\rank(G_m^S) \leq m \cdot \Lambda\left(n, q-1, \left\lfloor\frac{(q-1)n}{m}\right\rfloor\right)
&\leq m \cdot \min_{t \in (0,1)} \left(t^{-\lfloor\frac{(q-1)n}{m}\rfloor}\left(\frac{1-t^q}{1-t}\right)^n\right)\\
&\leq m \cdot \min_{t \in (0,1)}\left(t^{-\frac{q-1}{m}} \cdot \frac{1-t^q}{1-t}\right)^n.
\end{align*}

%WWWWWWWWWWWWWWWWWWWWWWWWWWWWWWWWWWWWWWWWWWWWWWWWWWWWWWWWWWWWWWWWWWWWWWWWWWWWWWWWWWWWWWWWWWWWW
%WWWWWWWWWWWWWWWWWWWWWWWWWWWWWWWWWWWWWWWWWWWWWWWWWWWWWWWWWWWWWWWWWWWWWWWWWWWWWWWWWWWWWWWWWWWWW

\section{Proof of lemma \ref{asymptotic}}

%(Note that we are using $\ds \frac{1-x^q}{1-x}$ as shorthand for $\sum_{i=0}^{q-1}x^i$, so $h_q$ is defined on $(0,\infty)$.)
To verify that the minimum at $x_0$ is well-defined, let $s = \frac{q-1}{m}$ and write
$$
h_q(x) = \left(\sum_{i = 0}^{\lfloor{s}\rfloor-1} x^{i-s} \right)+ \big(x^{\lfloor{s}\rfloor-s} + x^{\lfloor s\rfloor+1-s}\big) + \left(\sum_{i = \lfloor{s}\rfloor+2}^{q-1} x^{i-s}\right).
$$
Each function in the sum (counting $x^{\lfloor{s}\rfloor-s} + x^{\lfloor s\rfloor+1-s}$ as a single function) is convex. Therefore $h_q$ is also convex, meaning that anywhere its derivative vanishes on $(0,1)$ must be the unique minimum on that interval. Taking the derivative, we find
$$
h_q'(x) = \frac{x^{-\frac{q-1}{m}-1}}{m(1-x)^2} \cdot r_q(x)
$$
where
\begin{equation}\label{rdef}
r_q(x) = (q+m-1)x - (q-1) -x^q\big((q-1)(m-1)(1-x) + m\big).
\end{equation}
Given that
\begin{alignat*}{3}
&h_q'\left(\frac{q-1}{q+m-1}\right) &&= - \frac{q(q+m-1)}{m}\left(\frac{q-1}{q+m-1}\right)^{q-\frac{q-1}{m}-1} &&< 0\\
&h_q'(1) &&= \frac{q(q-1)(m-2)}{2m} &&> 0,
\end{alignat*}
there must indeed be a unique minimum occurring at some value $x_0$, and moreover, \\$\ds x_0 = \frac{q+\beta-1}{q+m-1}$ for some $\beta \in (0, m)$.
%\begin{align*}
%r_q\left(\frac{q-1}{q+m-1}\right) &= \left(\frac{q-1}{q+m-1}\right)^q\left(\frac{(q-1)^2(m-1)}{q+m-1} -qm+q-1\right)\\
%&= - \left(\frac{q-1}{q+m-1}\right)^q\left(\frac{qm^2}{q+m-1}\right) < 0

%\end{align*}
%while
%\begin{align*}
%r_q\left(\frac{q+1}{q+m-1}\right) &= 2+\left(\frac{q+1}{q+m-1}\right)^q\left(\frac{(q+1)(q-1)(m-1)}{q+m-1} -qm+q-1\right)\\
%&= 2- \left(1+\frac{m-2}{q+1}\right)^{-q-1}\left(2 + \frac{qm(m-2)}{q+m-1}\right)\left(1+\frac{m-2}{q+1}\right).
%\end{align*}
%Estimating the first factor via binomial expansion and combining the last two factors, we get
%$$
%r_q\left(\frac{q+1}{q+m-1}\right) >2- \left(1+m-2 + \frac{q(m-2)^2}{2(q+1)}\right)^{-1}\left(2 + \frac{2m-4}{q+1} + \frac{qm^2 -2qm}{q+1}\right) = 0.
%$$

\vspace{.2 in}
To get a better estimate for $\beta$, notice that
\begin{align*}
0 = r_q\left(\frac{q+\beta-1}{q+m-1}\right) %&= \beta + \left(\frac{q+\beta-1}{q+m-1}\right)^q\left[ \frac{q+\beta-1}{q+m-1}(q-1)(m-1) -qm+q-1\right]\\
 &= \beta - \left(1 +\frac{m-\beta}{q+\beta-1}\right)^{-q}\left[m^2 +\beta - m\beta - \frac{m(m-1)(m-\beta)}{q+m-1}\right] \\
& = \beta - \frac{m^2 - m\beta + \beta}{e^{m - \beta}}\big(1 - O(q^{-1})\big).
\end{align*}
Let $\ds f(x) = x - \frac{m^2 - mx+x}{e^{m - x}}$. We leave it to the reader to check that
\begin{itemize}
\item $f(x)$ has exactly one zero in $(0,1)$
\item $.25 < f'(x) < 1$ on $(0,1)$.
\end{itemize}
If $\alpha$ is that unique zero, then $f(\beta) = O(q^{-1})$ and $f(\alpha) = 0$, giving us
$$
\frac{f(\alpha)-f(\beta)}{\alpha-\beta} = \frac{O(q^{-1})}{\alpha-\beta}.
$$
Using the mean value theorem along with $.25 < f'(x) < 1$, we conclude that $\alpha = \beta + O(q^{-1})$. Therefore
$$
x_{0} = \frac{q+\beta-1}{q+m-1} = \frac{q+\alpha-1}{q+m-1} + O(q^{-2}).
$$ %or equivalently, $\ds\frac{e^{m-\alpha} - 1}{m} = \frac{m}{\alpha} -1$.
\vspace{.1 in}
To finish, we will estimate
$$
h_q\left(\frac{q+\alpha-1}{q+m-1} + O(q^{-2})\right).
$$
We can simplify this computation by rearranging the equation $r_q(x_0) = 0$ to get
$$
\frac{1-x_0^q}{1-x_0} = \frac{qm}{m + (q-1)(m-1)(1-x_0)},
$$
and thus
\begin{align*}
h_q\left(\frac{q+\alpha-1}{q+m-1} + O(q^{-2})\right) &= \left(\frac{q+\alpha-1}{q+m-1} + O(q^{-2})\right)^{-\frac{q-1}{m}} \cdot \frac{qm}{m + (q-1)(m-1)\left(1-\frac{q+\alpha-1}{q+m-1} - O(q^{-2})\right)}\\
& = \left(1 - \frac{m-\alpha-O(q^{-1})}{q+m-1}\right)^{-\frac{q-1}{m}} \cdot \frac{qm(q+m-1)(1 + O(q^{-2}))}{m(q+m-1) + (q-1)(m-1)(m-\alpha)}\\
& = e^{1 - \frac{\alpha}{m}}(1 + O(q^{-1})) \cdot \frac{qm(q+m-1)(1 + O(q^{-2}))}{m^2q - \alpha(m-1)(q-1)} \\ &= e^{1 - \frac{\alpha}{m}} \cdot \frac{mq}{m^2 - m\alpha+\alpha} + O(1).
\end{align*}

%WWWWWWWWWWWWWWWWWWWWWWWWWWWWWWWWWWWWWWWWWWWWWWWWWWWWWWWWWWWWWWWWWWWWWWWWWWWWWWWWWWWWWWWWWWWWW
%WWWWWWWWWWWWWWWWWWWWWWWWWWWWWWWWWWWWWWWWWWWWWWWWWWWWWWWWWWWWWWWWWWWWWWWWWWWWWWWWWWWWWWWWWWWWW

\pagebreak

\section{Estimating the size of $m$-general sets for certain $q$ and $m$}

Inequality \eqref{mubound} allows us to estimate $\mu_m(q)$ for large values of $q$. Table \ref{tablea} gives the asymptotic values for some small values of $m$. %We also compare these values to $\frac{eq}m$ to illustrate just how close it is to $h_q(x_0)$, especially for larger $m$.
These asymptotic estimates are useful when $q$ is a fixed large number, but we can compute the exact values of $\ds \min_{t \in (0,1)}\left(t^{-\frac{q-1}{m}} \cdot \frac{1-t^q}{1-t}\right)$ when $q$ is small. For instance, if $q= m = 3$, we can solve $r_3(x_0) = 0$ (see equation \eqref{rdef}) to get $x_0 = \frac{\sqrt{33}-1}{8}$. Theorem \ref{main} then recovers the same result as \cite{ElGi}, namely that a maximal cap in $AG(n,3)$ has size bounded above by
$$
6 + 3 \cdot (h_3(x_0))^n = O(2.756^n),
$$
or $\mu(3) < 0.923$.

Another particularly interesting case is $q = 2$, $m=4$, since $2$-flats in $AG(n,2)$ have exactly $4$ points. We find that the largest set $A \subset AG(n,2)$ in which no $2$-flat is ``fully covered" by points of $A$ has $M_3(n,2) < 8 + 4(1.755)^n$ points, hence $\mu_4(2) <  0.813$.

Table \ref{tableb} shows the upper bounds for $\mu_m(q)$ given by a direct calculation of
$$
\log_q\left(\min_{t \in (0,1)}\left(t^{-\frac{q-1}{m}} \cdot \frac{1-t^q}{1-t}\right)\right).
$$
Note that some boxes are unfilled because we did not obtain estimates in the cases where $q$ is even and $m$ is odd. The values marked with a `$*$' are those which can be improved to $0.5$ using the result from \cite{HTW}. Those marked with a `$\dagger$' were already derived in \cite{ElGi}.

\begin{center}

\begin{table}[!htb]
\captionsetup{font=normalsize}
\caption{Upper bounds on $\mu_m(q)$}
\begin{subtable}{1.5in}
%\vspace{.29 in}
\centering
\begin{tabular}{|c|c|}
\hline
$m$ & $\mu_m(q) < \ldots$\\ %& $\ds 1 - \log_q(e/m)$\\

 \hline
 %\hline

 $3$  & $1 - \log_q(1.188) $\\
  %\hline
 $4$ & $1 - \log_q(1.504)$\\
  %\hline
 $5$  & $1 - \log_q(1.853)$\\
  %\hline
$6$ & $1 - \log_q(2.212)$\\
  %\hline
 $7$ &  $1 - \log_q(2.577)$\\
  %\hline
	 $8$  &  $1 - \log_q(2.944)$\\
  \hline
\end{tabular}
\caption{Bounds for small $m$ and sufficiently large $q$}\label{tablea}
\end{subtable}
\begin{subtable}{5in}
\vspace{-.29 in}
\centering
\begin{tabular}{c|c|c|c |c| c |c|c|c|c|}
%\cline{3-10}
\multicolumn{2}{c}{} & \multicolumn{8}{c}{$q$}\\
 \cline{3-10}
  \multicolumn{2}{c|}{}& 2 & 3 & 4 & 5 & 7 & 8 & 9 & 11 \\

 \cline{2-10}
% \hline

\multirow{6}{*}{$m$}
 &$3$ &    N/A  & $0.923^\dagger$ &         &$0.930^\dagger$ & $0.935^\dagger$ &         & $0.938^\dagger$ & $0.941^\dagger$\\ 
  %\cline{2-10}
& $4$ & $0.813$ & $0.821^*$ & $0.829$ &$0.836$ & $0.846$ & $0.851$ & $0.854$ & $0.861$\\
  %\cline{2-10}
 &$5$ &         & $0.735^*$ &         &$0.756$ & $0.771$ &         & $0.782$ & $0.791$\\
  %\cline{2-10}
&$6$  & $0.651$ & $0.665^*$ & $0.679$ &$0.690$ & $0.708$ & $0.716$ & $0.722$ & $0.734$\\
  %\cline{2-10}
 &$7$ &         & $0.609^*$ &         &$0.636$ & $0.657$ &         & $0.673$ & $0.685$\\ 
%\cline{2-10}
 &$8$ & $0.544$ & $0.562^*$ & $0.577$ &$0.591$ & $0.613$ & $0.622$ & $0.631$ & $0.644$\\ 
  %\hline
	\cline{2-10}
\end{tabular}
\caption{Bounds for specific small $m$ and $q$}\label{tableb}
\end{subtable}
\end{table}

\end{center}

\pagebreak

%WWWWWWWWWWWWWWWWWWWWWWWWWWWWWWWWWWWWWWWWWWWWWWWWWWWWWWWWWWWWWWWWWWWWWWWWWWWWWWWWWWWWWWWWWWWWW
%WWWWWWWWWWWWWWWWWWWWWWWWWWWWWWWWWWWWWWWWWWWWWWWWWWWWWWWWWWWWWWWWWWWWWWWWWWWWWWWWWWWWWWWWWWWWW

\section*{Acknowledgments}
I would like to thank the anonymous reviewers for their many helpful suggestions.

\end{document}